\documentclass[12pt]{amsart}

\usepackage[dvips]{graphicx}
	\usepackage{amsmath,amssymb,amsthm,wrapfig,amsfonts,enumerate,latexsym}

\newtheorem{theorem}{Theorem}[section]
\newtheorem{lemma}[theorem]{Lemma}
\newtheorem{corollary}[theorem]{Corollary}
\newtheorem{proposition}[theorem]{Proposition}

\theoremstyle{definition}

\newtheorem{definition}[theorem]{Definition}
\newtheorem{remark}[theorem]{Remark}
\newtheorem*{acknowledgement}{Acknowledgements}       
 
\numberwithin{equation}{section}

\begin{document}

\title[CAT(0) spaces on which a certain type of singularity is bounded]{CAT(0) spaces on which a certain type of singularity is bounded }
\author[Tetsu Toyoda]{Tetsu Toyoda}
\address{Graduate School of Mathematics\\ 
Nagoya University, Chikusa-ku Nagoya 464-8602 Japan}
\email{tetsu.toyoda@math.nagoya-u.ac.jp}

\maketitle

\begin{center}
(to appear in Kodai Math. J.  \textbf{33} (2010) )
\end{center}


\begin{abstract}
In this paper, we will consider a family $\mathcal{Y}$ of complete $\mathrm{CAT}(0)$ spaces 
such that the tangent cone $TC_p Y$ at each point $p\in Y$ of each $Y\in\mathcal{Y}$ 
is isometric to a (finite or infinite) product of the Euclidean cones $\mathrm{Cone}(X_{\alpha})$ 
over elements $X_{\alpha}$ of some 
Gromov-Hausdorff precompact family $\{ X_{\alpha}\}$ of $\mathrm{CAT}(1)$ spaces. 
Each element 
of such $\mathcal{Y}$ 
is a space presented by Gromov \cite{G} as an example of a 
``CAT(0) space with 
``bounded" singularities". 
We will show that the Izeki-Nayatani invariants of 
spaces in such a family are uniformly 
bounded from above by a constant strictly less than $1$.
\end{abstract}
\setlength{\baselineskip}{5mm}

\section{Introduction}
\label{intro}

In \cite{G}, Gromov introduced the term 
``CAT(0) space with 
`bounded' singularities", and 
remarked that there exist infinite groups 
which admit no uniform embeddings 
into such a space. 
He used this terminology 
without providing its precise definition, 
but as examples of such spaces, he presented 
$\mathrm{CAT}(0)$ spaces $Y$ 
such that the tangent cone $TC_p Y$ at each point $p\in Y$ is isometric to a (finite or infinite) 
product of Euclidean cones $\mathrm{Cone}(X_{\alpha})$ over elements $X_{\alpha}$ of some
Gromov-Hausdorff precompact family $\{ X_{\alpha}\}$ of $\mathrm{CAT}(1)$ spaces.

On the other hand, 
Izeki and Nayatani \cite{IN} defined an invariant 
$\delta (Y)\in\lbrack 0,1\rbrack$
of 
a complete $\mathrm{CAT} (0)$ space $Y$. 
And some general results for 
$\mathrm{CAT}(0)$ spaces whose Izeki-Nayatani invariants are bounded from above 
were proved by 
Izeki, Kondo, and Nayatani 
(\cite{IN}, \cite{IKN}, \cite{IKN2}, \cite{K}, \cite{K2}). 
Group $\Gamma$ is said to have the \textit{fixed-point property} for a metric space $Y$, 
if for any group homomorphism $\rho :\Gamma\to\mathrm{Isom}(Y)$ 
there exists a point $p\in Y$ such that 
$\rho (\gamma )p =p$ for all $\gamma\in\Gamma$. 
Izeki, Kondo and Nayatani \cite{IKN2} proved that 
a random group of Gromov's graph model has the fixed-point property for all elements $Y$ of 
a family $\mathcal{Y}$ of $\mathrm{CAT}(0)$ spaces whose Izeki-Nayatani invariants are 
uniformly bounded from above by a constant strictly less than $1$: 
$$
\sup\{\delta (Y)\hspace{1mm}|\hspace{1mm} Y\in\mathcal{Y}\}<1 . 
$$
Moreover, it is straightforward to see that an expander admits no 
uniform embedding into a complete CAT(0) space $Y$ with $\delta(Y)<1$
(see \cite{K2}). 
Combining this with Gromov's argument in \cite{G},  
the existence of infinite groups which admit 
no uniform embeddings into a space $Y$ with $\delta (Y)<1$ follows. 
This seems to suggest that 
the Izeki-Nayatani invariant 
measures a certain type of ``singularity" 
similar to Gromov's notion. 

Although these general results were proved, 
the computation of the Izeki-Nayatani invariant is difficult. 
It is still unclear what kind of $\mathrm{CAT}(0)$ spaces $Y$ or 
families  $\mathcal{Y}$ of $\mathrm{CAT}(0)$ spaces have 
the boundedness property as above. 
It had been even unknown whether there exists a complete CAT(0) space $Y$
with $\delta (Y)=1$ or not, until 
Kondo \cite{K2} showed the existence of CAT(0) spaces with $\delta =1$
fairly recently.

In this paper, we prove the following theorem.

\begin{theorem}\label{main}
Let $\mathcal{Y}$ be 
a family of complete $\mathrm{CAT}(0)$ spaces 
such that the tangent cone $TC_p Y$ at each point $p\in Y$ on each $Y\in\mathcal{Y}$ 
is isometric to a (finite or infinite) product of the Euclidean cones $\mathrm{Cone}(X_{\alpha})$ 
over elements $X_{\alpha}$ of some 
Gromov-Hausdorff precompact family $\{ X_{\alpha}\}$ of complete $\mathrm{CAT}(1)$ spaces. 
Then we have 
$$
\sup_{Y\in\mathcal{Y}}\delta (Y)<1 .
$$
\end{theorem}

Here, we use the word \textit{product} of Euclidean cones 
$T_1 ,T_2 , \ldots$ 
in the sense of $\ell^2$-product of the pointed metric spaces $(T_1 ,O_1 ), (T_2 ,O_2 ),\ldots$, 
where each $O_n$ is the cone point of $T_n$. 
That is, the product $T$ of the cones $T_1 ,T_2 ,\ldots$ 
consists of all sequences $(x_n )_n$ such that $x_n \in T_n$ and 
$\sum_{n}d_n (O_n ,x_n )^2 <\infty$, 
and $T$ is equipped with the metric function $d$
defined by 
$$
d(x,y )^2 =\sum_{n=1}^{\infty}d_n (x_n , y_n )^2
$$
for any $x=(x_1 ,x_2 ,\ldots )\in T$ and any $y=(y_1 ,y_2 ,\ldots )\in T$, 
where $d_n$ is the metric function on $T_n$ for each $n$. 
Then, 
$T$ also has a cone structure with the cone point 
$O=(O_1 ,O_2 ,\ldots )$. 
And completeness and $\mathrm{CAT}(0)$ condition are preserved by this construction. 

Combining Theorem \ref{main} with the general results mentioned above, we have the following corollary. 

\begin{corollary}
{\boldmath $\mathrm{(i)}$}\hspace{1mm}
If $Y$ is 
a complete $\mathrm{CAT}(0)$ space 
such that the tangent cone at each point $y\in Y$ is isometric to a (finite or infinite)
product of Euclidean cones $\mathrm{Cone}(X_{\alpha})$ over elements $X_{\alpha}$ of some
Gromov-Hausdorff precompact family $\{ X_{\alpha}\}$ of $\mathrm{CAT}(1)$ spaces, 
then there exists infinite groups which admit no uniform embeddings into $Y$. 
\hspace{1.2mm}{\boldmath $\mathrm{(ii)}$}\hspace{1mm}
There exist infinite groups which has the fixed-point property for all
elements $Y$ in such a family $\mathcal{Y}$ as in Theorem 1.1.
\end{corollary}

Here, 
\textbf{(i)} has already been remarked in \cite{G}. 
And \textbf{(ii)} follows from the general result in \cite{IKN2}. 
\textbf{(ii)} can be stated in terms of random groups(see \cite{IKN2}).

In the end of this paper, we claim that 
by the same technique used in the proof of Theorem \ref{main}, 
we can prove a more general statement, 
which includes Theorem \ref{main} as a special case (Proposition \ref{sufficient}).

\section{Preliminaries on $\mathrm{CAT}(0)$ spaces}

In this section we recall some basic definitions and facts 
concerning 
$\mathrm{CAT}(0)$ spaces. 
For a detailed exposition, 
we refer the reader to \cite{BH}, \cite{BBI} or \cite{S}. 

For $\kappa >0$ let $M_{\kappa}^2$ denote the 
simply connected, complete 2-dimensional Riemannian manifold of 
constant Gaussian curvature $\kappa$, and let  
$d_{\kappa}$ be its distance function. 
Let $D_{\kappa}\in (0,\infty\rbrack$ be the 
diameter of $M_{\kappa}^2$. 

Let $(Y,d_Y)$ be a metric space. 
A \textit{geodesic} in $Y$  
is an isometric embedding $\gamma$ of a 
closed interval $\lbrack a,b\rbrack$ into $Y$. 
A \textit{geodesic triangle} in $Y$ is a triple 
$\triangle =(\gamma_1 ,\gamma_2 ,\gamma_3 )$ 
of geodesics 
$\gamma_i :\lbrack a_i,b_i \rbrack\to Y$ 
such that 
$$
\gamma_1 (b_1)=\gamma_2 (a_2),\quad\gamma_2 (b_2)=\gamma_3 (a_3), 
\quad\gamma_3 (b_3)=\gamma_1 (a_1) .
$$
If $\triangle$ has 
a perimeter less than $2 D_{\kappa}$: 
$\sum_{i=1}^3 |b_i -a_i | <2 D_{\kappa}$,  
then
there is a geodesic triangle 
$$
\triangle^{\kappa}=
(\gamma^{\kappa}_1 ,\gamma^{\kappa}_2 ,
\gamma^{\kappa}_3 ),\quad
\gamma_i :\lbrack a_i,b_i \rbrack\to M_{\kappa}^2
$$
in $M_{\kappa}^2$
, which has the same side lengths as $\triangle$.  
This triangle $\triangle^{\kappa}$ is unique up to isometry of $M_{\kappa}^2$, 
and 
we call it the \textit{comparison triangle} of $\triangle$ in 
$M_{\kappa}^2$. 
Then $\triangle$ is said to be $\kappa$-\textit{thin} 
if 
$$
d_Y (\gamma_i (s),\gamma_j (t))\leq
d_{\kappa}(\gamma^{\kappa}_i (s),\gamma^{\kappa}_j (t))
$$
whenever $i,j\in\{ 1,2,3\}$ and $s\in\lbrack a_i ,b_i \rbrack$, 
and $t\in\lbrack a_j ,b_j \rbrack$.

\begin{definition}
A metric space $(Y,d)$ is called a $\mathit{CAT (\kappa )}$ \textit{space}, 
if for any pair of points $p,q\in Y$ 
with $d(p,q)<D_{\kappa}$ 
there exists a geodesic from $p$ to $q$, 
and any geodesic triangle in $Y$ with 
perimeter$< 2D_{\kappa}$ 
is $\kappa$-thin. 
\end{definition}

Next, we recall the definition of the Euclidean cone. 
Let $(X,d_X )$ be a metric space. 
The cone $\mathrm{Cone}(X)$ over $X$ is 
the quotient of the product $X\times\lbrack 0,\infty )$ 
obtained by identifying all points in 
$X\times\{ 0\}\subset X\times\lbrack 0,\infty )$. 
The point represented by $(x,0)$ is called the \textit{cone point} of $\mathrm{Cone}(X)$ and 
we will denote this point by $O_{\mathrm{Cone}(X)}$ in this paper.  
The cone distance $d_{\mathrm{Cone}(X)} (v,w)$ 
between two points $v,w\in\mathrm{Cone}(X)$ 
represented by $(x,t),(y,s)\in X\times\lbrack 0,\infty )$ 
respectively, is defined by 
$$
d_{\mathrm{Cone}(X)} (v,w)= 
\sqrt{t^2 +s^2 -2ts\cos (\min\{\pi ,d_X (x,y)\})} .
$$
Then $(\mathrm{Cone}(X),d_{\mathrm{Cone}(X)})$ is a metric space, and we call it  
the \textit{Euclidean cone} over $(X,d_X )$. 
It is known that a metric space 
$(X,d_X )$ is a $\mathrm{CAT}(1)$ space 
if and only if 
$(\mathrm{Cone}(X),d_{\mathrm{Cone}(X)})$ 
is a $\mathrm{CAT}(0)$ space. 

Suppose that $Y$ is a $\mathrm{CAT}(0)$ space. 
Then by the definition of $\mathrm{CAT}(0)$ space, there is a unique geodesic joining any pair of points in $Y$. 
So, for any triple of points $(p,q,r)$ in $Y$, 
it makes sense to denote by $\triangle (p,q,r)$ 
the geodesic triangle consisting of 
three geodesics joining each pair of the three points. 

Let $\gamma :\lbrack a,b\rbrack\to Y$, 
$\gamma' :\lbrack a',b' \rbrack\to Y$
be two geodesics in a $\mathrm{CAT}(0)$ space $Y$ such that 
$$
\gamma (a)=\gamma' (a')=p\in Y .
$$ 
We define the \textit{angle} $\angle_p (\gamma ,\gamma' )$ 
between $\gamma$, $\gamma'$ as 
$$
\angle_p (\gamma ,\gamma' )=
\lim_{t\to a,t'\to a'}\angle^{0}_p (\gamma(t),\gamma(t') ) ,
$$ 
where $\angle^{0}_p (\gamma(t),\gamma(t') ) $ is 
the corresponding angle of 
the comparison triangle of
\linebreak
$\triangle (p,\gamma (t),\gamma' (t'))$ in $M_{0}^2 =\mathbb{R}^2$. 
The existence of the limit follows from the definition of 
$\mathrm{CAT}(0)$ space.

\begin{definition}
Let $(Y,d_Y )$ be a complete 
$\mathrm{CAT}(0)$ space, and let $p\in Y$. 
We denote by $(S_p Y )^{\circ}$ the set of all geodesics 
$\gamma :\lbrack a,b\rbrack\to Y$ 
such that $\gamma (a)=p$. 
Then the angle $\angle_p$ defines a pseudometric on 
$(S_p Y )^{\circ}$. The \textit{space of directions} $S_p Y$ at $p$ is 
the metric completion of the quotient space of $(S_p Y )$ where we identify any $x,y\in S_p Y$ 
with $\angle_p (x,y)=0$. 
We define the \textit{tangent cone} $TC_p Y$ of $Y$ at $p$ to be the 
Euclidean cone $\mathrm{Cone}(S_p Y )$ over the 
space of directions at $p$.  
\end{definition}

If $(Y,d_Y )$ is a complete $\mathrm{CAT}(0)$ space and if $p\in Y$, 
then it can be proved that the space of directions $S_p Y$ at $p$ is a complete $\mathrm{CAT}(1)$ space. 
Hence, the tangent cone $TC_p Y$ at $p$ is a complete $\mathrm{CAT}(0)$ space. 

Finally, we recall some basic notions and facts about probability measures on a metric space 
$(Y,d_Y )$. 
In this paper, we will treat only finitely supported measures. 
Measure $\nu$ on $Y$ is 
\textit{finitely supported} if there exists 
a finite subset $S\subset Y$ such that $\nu (Y\backslash S)=0$. 
We call the minimal subset $S$ with such a property 
the \textit{support} of $\nu$, and denote it by $\mathrm{supp}(\nu )$. 
We denote by 
$\mathcal{P}(Y)$ 
the set of all finitely supported 
probability measures on $Y$. 
If $\mathrm{supp}(\nu )=\{p_1 ,\ldots , p_n \}$, 
then $\nu$ can be represented as 
\begin{equation}\label{dirac-form}
\nu=\sum_{i=1}^n t_i \mathrm{Dirac}_{p_i}
\end{equation}
by nonnegative real numbers $t_1 ,\ldots , t_n$ with 
$\sum_{i=1}^n t_i=1$, where $\mathrm{Dirac}_{p_i} $ stands for the Dirac measure at $p_i \in Y$. 
We will also use the notation $\mathcal{P}'(Y)$ 
to denote the subset of $\mathcal{P}(Y)$ 
consisting of all measures 
whose supports contain at least two points. 
Let $Z$ be a set and let $\phi :Y\to X$ be a map. 
Then for any $\nu\in\mathcal{P}(Y)$, we define the \textit{pushforward} measure $\phi_* \mu$ on $X$ as 
$$
\phi_* \nu (A)=\mu \left(\phi^{-1}(A)\right) ,\quad A\subset X
$$
If we write $\nu$ as in the form \eqref{dirac-form}, we can write $\phi_* \nu$ as 
$$
\phi_*\nu=\sum_{i=1}^n t_i \mathrm{Dirac}_{\phi (p_i )}
$$
If $(Y,d_Y )$ is a complete $\mathrm{CAT}(0)$ space, 
and if $\nu\in\mathcal{P}(Y)$, 
there exists 
a unique point $\mathrm{bar}(\nu )\in Y$ which minimizes the function 
$$
y\mapsto 
\int_Y d(y,z)^2 \nu (dz)
$$
defined on $Y$. 
This point is called the 
\textit{barycenter} of $\nu$. 
We refer the reader to \cite{S} for the existence and uniqueness of barycenter.

\section{Hilbert sphere valued maps and 
an invariant of a $\mathrm{CAT}(1)$ space}

In this section, we define a certain invariant of complete $\mathrm{CAT}(1)$ spaces. 
First we set up some notations for Hilbert sphere valued maps on $\mathrm{CAT} (1)$ spaces. 
Let $\mathcal{H}$ be a real Hilbert space, 
and let $\phi :X\to\mathcal{H}$ be a map whose image is contained in the unit sphere in $\mathcal{H}$. 
Thus 
$\|\phi (x)\|=1$ for all $x\in X$.  
Let $\mu\in\mathcal{P}(X)$ 
be a finitely supported probability measure on $X$. 
We define the vector $\mathbb{E}_{\mu}\lbrack\phi\rbrack\in\mathcal{H}$ as 
$$ 
\mathbb{E}_{\mu}\lbrack\phi\rbrack
=
\int_{X} \phi (x ) \mu (dx) .
$$
And 
if the vector $\mathbb{E}_{\mu} \lbrack\phi\rbrack$ 
is not the zero vector,  
we denote by 
$\tilde{\mathbb{E}}_{\mu}\lbrack\phi\rbrack$ 
the unit vector parallel to $\mathbb{E}_{\mu}\lbrack\phi\rbrack$: 
$$
\tilde{\mathbb{E}}_{\mu} \lbrack\phi\rbrack =\frac{1}{\| \mathbb{E}_{\mu}\lbrack\phi\rbrack\|}
\mathbb{E}_{\mu} \lbrack\phi\rbrack .
$$
Then the value 
$\|\mathbb{E}_{\mu} \lbrack\phi\rbrack\| \in\lbrack 0,1\rbrack$ 
amounts to a sort of concentration 
of the pushforward measure $\phi_* \mu$ around $\tilde{\mathbb{E}}_{\mu} \lbrack\phi\rbrack$ 
on the unit sphere. 
By simple calculation, we have 
\begin{equation}\label{variance}
\left\|\mathbb{E}_{\mu} \lbrack\phi\rbrack\right\|
=
\int_{X}\langle\tilde{\mathbb{E}}_{\mu} \lbrack\phi\rbrack ,
\phi (x ) \rangle \mu (dx ) 
\end{equation}
whenever $\|\mathbb{E}_{\mu} \lbrack\phi\rbrack\|\neq 0$. 

Now we define an invariant of a complete $\mathrm{CAT}(1)$ space 
by using the notations introduced above. 
This invariant is designed for estimating 
the Izeki-Nayatani invariant of a $\mathrm{CAT}(0)$ space, 
whose definition will be recalled in the next section. 

\begin{definition}
Let $(X,d_X )$ be a metric space, and let 
$\mu\in\mathcal{P}(X)$. 
We define $\tilde{\delta}(\mu )\in\lbrack 0,1\rbrack$ 
to be 
$$
\tilde{\delta}(\mu )
=
\inf_{\phi}\|\mathbb{E}_{\mu}\lbrack\phi\rbrack\|^2 ,
$$
where the infimum is taken over all maps $\phi :X\to\mathcal{H}$ 
to some Hilbert space $\mathcal{H}$ such that
\begin{equation}\label{realization-tilde}
\|\phi (x)\| =1,\quad
\angle\left(\phi (x),\phi (y)\right)
\leq d_X (x,y)
\end{equation}
for any $x,y\in X$. 
Here and henceforth, we denote the angle between two vectors $v,w$ in 
any Hilbert space by $\angle (v,w)$. 

Suppose $(X,d_X)$ is a complete $\mathrm{CAT}(1)$ space and  
$\iota :X\to\mathrm{Cone}(X)$ is the canonical 
inclusion of $X$ into its Euclidean cone.  
Then, we define $\tilde{\delta}(X)$ to be
$$
\tilde{\delta}(X)
=
\sup\{ \tilde{\delta}(\mu )\hspace{1mm}|\hspace{1mm}
\mu\in\mathcal{P}(X),
\mathrm{bar}(\iota_{*}\mu )=O_{\mathrm{Cone}(X)}
\} .
$$
When there is no measure satisfying such a condition, 
we define $\tilde{\delta}(X)=-\infty$. 
\end{definition}

To estimate this invariant in the proceeding sections, 
we will use the following fact: 
\begin{lemma}\label{barycenter}
Let $(X,d_X )$ be a complete $\mathrm{CAT}(1)$ space. 
For $v,w\in\mathrm{Cone}(X)$ represented by $(x,t),(y,s)\in X\times\mathbb{R}$ respectively, 
we set 
$$
\langle v,w\rangle
=
ts\cos\left(\min\{\pi ,d_X (x,y)\}\right) . 
$$
Then for any $\nu\in\mathcal{P}(\mathrm{Cone}(X))$ 
the following two conditions are equivalent: 
\begin{description}
\item[(i)]
$\mathrm{bar}(\nu )=O_{\mathrm{Cone}(X)}$.
\item[(ii)]
$
\int_{\mathrm{Cone}(X)}
\langle E_x ,v\rangle\hspace{1mm}\nu (dv)\leq 0
$, whenever $x\in X$ and $E_x$ is an element of $\mathrm{Cone}(X)$ represented by 
$(x,1)$. 
\end{description}
\end{lemma}

\begin{proof}
For $w\in\mathrm{Cone}(X)$ represented by $w=(y,s)\in X\times\mathbb{R}$, 
we write $\| w\| =s$. 
Fix $x\in X$ and let $v_t$ be an element of $\mathrm{Cone}(X)$ represented by 
$(x,t)\in X\times\mathbb{R}$. 
Suppose that $\mathrm{bar}(\nu )=O_{\mathrm{Cone}(X)}$. 
Then 
the function 
\begin{align}
\label{Ft}
F_x (t)
&=
\int_{\mathrm{Cone}(X)} d_{\mathrm{Cone}(X)}(v_t ,w)^2 \nu (dw) \\
&=
\int_{\mathrm{Cone}(X)}
\big\{
t^2 +\| w\|^2 
-2t\langle E_x ,w\rangle
\big\}
\nu (dw), \nonumber
\end{align}
defined on $\lbrack 0,\infty )$
must 
attain its minimum at $t=0$.  
This happens if and only if 
$$
F'_{x}(t)
=
2\left( t-\int_{\mathrm{Cone}(X)}\langle E_x ,w\rangle\hspace{1mm}\nu (dw)\right) \geq 0.
$$
for all $t\in\mathbb{R}$. 
So \textbf{(ii)} follows. 

Conversely, if \textbf{(ii)} holds,  
then  
the function $F_x$ on $\lbrack 0,\infty )$ as \eqref{Ft} attains its minimum 
at $t=0$ for each $x\in X$. 
And it is easily seen that $\mathrm{bar}(\nu )=O_{\mathrm{Cone}(X)}$. 
\end{proof}

In the final section, we will use this lemma in the following form. 

\begin{corollary}\label{pi/3-2/3}
Let $(X,d_X )$ be a complete $\mathrm{CAT}(1)$ space, and let $\iota :X\to\mathrm{Cone}(X)$ 
be the canonical inclusion. 
If $\mu\in\mathcal{P}(X)$ satisfies 
$\mathrm{bar}(\iota_* \mu )=O_{\mathrm{Cone}(X)}$, then 
we have 
$$
\mu \left(
\left\{
y\in X\hspace{1mm}\Big|\hspace{1mm}
d_X (x,y)\leq\theta
\right\}
\right)
\leq\frac{1}{1+\cos\theta}
$$
for any $x\in X$ and any $0\leq\theta <\frac{\pi}{2}$. 
In particular, we have 
$$
\mu \left(
\left\{
y\in X\hspace{1mm}\Big|\hspace{1mm}
d_X (x,y)\leq\frac{\pi}{3}
\right\}
\right)
\leq\frac{2}{3}
$$
for all $x\in X$. 
\end{corollary}

\begin{proof}
Suppose there is $x_0 \in X$ such that
$$
\mu \left(
\left\{
y\in X\hspace{1mm}\Big|\hspace{1mm}
d_X (x_0 ,y)\leq\theta
\right\}
\right)
>\frac{1}{1+\cos\theta}. 
$$
Then we would have 
\begin{align*}
\int_{X}\cos\left( \min\{\pi ,d_X (x_0 ,x)\}\right) \mu (dx)
&=
\int_{\{x\in X\hspace{0.5mm}|\hspace{0.5mm}d_X (x,x_0 )\leq\theta\}}
\cos\left( \min\{\pi ,d_X (x_0 ,x)\}\right) \mu (dx )
\\
&\hspace{5mm}
+
\int_{X\backslash\{x\in X\hspace{0.5mm}|\hspace{0.5mm}d_X (x,x_0 )\leq\theta\}}
\cos\left( \min\{\pi ,d_X (x_0 ,x)\}\right) \mu (dx )
\\
&>
\cos\theta\times\frac{1}{1+\cos\theta}+(-1)\times\left( 1-\frac{1}{1+\cos\theta}\right)
\\
&=0. 
\end{align*}
This implies $\mathrm{bar}(\iota_{*}\mu )\neq O_{\mathrm{Cone}(X)}$ by Lemma \ref{barycenter}, 
which is a contradiction. 
\end{proof}

\section{Izeki-Nayatani invariant}

In this section, 
we recall the definition of the invariant $\delta$ of a complete $\mathrm{CAT}(0)$ space 
introduced by Izeki and Nayatani \cite{IN}. 
We will then 
derive a relation between $\delta$ and 
the invariant $\tilde{\delta}$ of a complete $\mathrm{CAT}(1)$ space 
defined in the previous section. 
More information about the Izeki-Nayatani invariant $\delta$ can be found in 
\cite{IN}, \cite{IKN}, \cite{IKN2}, \cite{K} and \cite{P}. 

\begin{definition}[\cite{IN}]
Let $(Y,d_Y )$ be a complete $\mathrm{CAT}(0)$ space. 
Recall that $\mathcal{P}'(Y)$ is 
the subset of $\mathcal{P}(Y)$ 
consisting of all measures 
whose supports contain at least two points. 
For any $\nu\in\mathcal{P}'(Y)$, 
we define $\delta (\nu )$ to be 
$$
\delta (\nu )=\inf_{\phi }
\frac{\|\int_{Y}\phi (p)\nu (dp)\|^2}{\int_{Y}\|\phi (p)\|^2 \nu (dp)},
$$
where the infimum is taken over all maps $\phi :\mathrm{supp} (\nu )\to\mathcal{H}$ 
from the support of $\nu$ to some Hilbert space $\mathcal{H}$ such that 
\begin{align}
&\|\phi (p)\| =d(\mathrm{bar}(\nu ),p),\label{umbrellaborn}\\
&\|\phi (p)-\phi (q)\| \leq d(p,q) 
\label{1-lipschitz}
\end{align}
for all $p,q\in\mathrm{supp}(\nu )$. 
Then the Izeki-Nayatani invariant $\delta (Y)$ of $Y$ is defined by 
$$
\delta (Y)=\sup\left\{\delta (\nu)\hspace{1mm}|\hspace{1mm}
\nu\in\mathcal{P}'(Y)\right\}. 
$$
\end{definition}

By definition, we have $0\leq\delta (\nu)\leq 1$ and $0\leq\delta (Y)\leq 1$. 
When $Y$ is a Euclidean cone, 
we define $\delta (Y, O_Y )\in\lbrack 0,1\rbrack$ to be
$$
\delta (Y, O_Y )=
\sup\left\{\delta (\nu)\hspace{1mm}|\hspace{1mm}
\nu\in\mathcal{P}'(Y), \mathrm{bar}(\nu )=O_Y 
\right\} ,
$$ 
where $O_Y$ is the cone point of $Y$. 
When there is no measure satisfying such a condition, we define $\delta (Y, O_Y)=-\infty$. 
The following lemma is shown in \cite{IN}. 

\begin{lemma}[\cite{IN}]\label{conedelta}
Suppose that $Y$ is a complete $\mathrm{CAT} (0)$ space, and 
$\nu\in\mathcal{P}'(Y)$. 
Then we have 
$$
\delta (\nu)\leq\delta (TC_{\mathrm{bar}(\nu )}Y,\hspace{0.8mm}O_{TC_{\mathrm{bar}(\nu )}Y}) .
$$
In particular, we have 
$$
\delta (Y)\leq\sup\{\delta (TC_p Y,\hspace{0.8mm} O_{TC_p Y})\hspace{1mm}|\hspace{1mm}
p\in Y\} .
$$
\end{lemma}

The following lemma is a slight generalization of Proposition 6.5 in \cite{IN}. 

\begin{lemma}\label{product}
Let $(T_1 ,d_1 ) ,(T_2 ,d_2 ) ,(T_3 ,d_3 ) ,\ldots $ be complete $\mathrm{CAT}(0)$ spaces which are 
isometric to Euclidean cones,  
and let $O_1 ,O_2 ,\ldots$ be their cone points respectively. 
Let $T$ be the cone obtained as the product of $T_1 , T_2 ,\ldots$ with the cone point 
$O=(O_1 ,O_2 ,\ldots )$. 
Then we have
$$
\delta (T, O) =\sup_n \delta (T_n,\hspace{0.8mm} O_{n}) .
$$
\end{lemma}

\begin{proof}
The following proof is almost the same argument as in the proof of Proposition 6.5 in \cite{IN}. 
We however include it  for the sake of completeness. 

First, the inequality $\delta (T, O) \geq\sup_n \delta (T_n,\hspace{0.5mm} O_{n}) $ is obvious. 
Because 
we have the canonical isometric embedding $\mathcal{I}_n : T_n \to T$ for each $n$,  
and for each $\mu\in\mathcal{P}'(T_n )$ with $\mathrm{bar}(\mu )=O_n$, 
it is easy to see that $\mathrm{bar}(\mathcal{I}_{n*} \mu )=O$ and 
$\delta (\mu )=\delta (\mathcal{I}_{n*} \mu )$. 

Let 
\begin{equation*}
\mu =\sum_{i=1}^m t_i \mathrm{Dirac}_{v_i}
\in\mathcal{P}'(T)
\end{equation*}
be an arbitrary measure in $\mathcal{P}'(T)$ with $\mathrm{bar}(\mu )=O$, 
where 
$v_1 ,\ldots ,v_m \in T$ and 
$t_1 ,\ldots ,t_m >0$ with $\sum_{i=1}^m t_i =1$. 
Write $v_i=(v_i^{(1)},v_i^{(2)},\ldots )$ and let 
$$
\mu_n
=
\sum_{i=1}^m t_i \mathrm{Dirac}_{v_i^{(n)}}
\in\mathcal{P}'(T_n), \quad n=1,2,\ldots .
$$
Then $\mathrm{bar}(\mu_n)=O_n$ for each $n$. 
Because if we have $\mathrm{bar}(\mu_n )\neq O_{n}$ for some $n$, 
it is easy to show that
$$
\int_{T}d(w, B)^2 \mu (dw)
<
\int_{T}d(w, O)^2 \mu (dw) ,
$$
where $B\in T$ is a point in $T$ such that all of its components 
are the cone points but $\mathrm{bar}(\mu_n )$ for the $n$-th component, 
and it contradicts the assumption that $\mathrm{bar}(\mu )=O$. 

Let $\varepsilon >0$ be an arbitrary positive number. 
By the definition of $\delta (T_n ,O_n)$, there exists 
a map $\phi_n :\mathrm{supp}(\mu_n )\to\mathcal{H}_n$ from the support of 
$\mu_n$ to some Hilbert space $\mathcal{H}_n$ with 
the properties \eqref{umbrellaborn} and \eqref{1-lipschitz} with respect to $\mu_n$, satisfying 
$$
\frac{\|\int_{T_n} \phi_n (v) \mu_n (dv)\|^2}{\int_{T_n} \|\phi_n (v)\|^2 \mu_n (dv)}
\leq
\delta (T_n ,O_n) +\varepsilon .
$$
We define a map $\phi :\mathrm{supp}(\mu )\to\mathcal{H}$ from 
the support of $\mu$ to the Hilbert space 
$\mathcal{H}=\mathcal{H}_1 \oplus\mathcal{H}_2 \oplus\cdots$ to be 
$$
\phi (v_i )=\left(\phi_1 (v_i^{(1)}),\phi_2 (v_i^{(2)}),\ldots \right) ,
\quad i=1,\ldots ,m .
$$
Then it is straightforward to see that $\phi$ satisfies the properties 
\eqref{umbrellaborn} and \eqref{1-lipschitz} with respect to $\mu$. 
And we have 
\begin{multline*}
\delta(\mu )\leq
\frac{\|\int_T \phi (v)\mu (dv)\|^2}{\int_T \|\phi (v )\|^2 \mu (dv)}
=
\frac{\sum_{n=1}^{\infty}\|\sum_{i=1}^m t_i \phi_n (v_i^{(n)} )\|^2}{
\sum_{n=1}^{\infty}\sum_{i=1}^m t_i \|\phi_n (v_i^{(n)} )\|^2}
\\
\leq
\sup_n \frac{\|\sum_{i=1}^m t_i \phi_n (v_i^{(n)} )\|^2}{\sum_{i=1}^m t_i \|\phi_n (v_i^{(n)} )\|^2}
\leq
\sup_n \left(\delta (T_n ,O_n) +\varepsilon \right).
\end{multline*}
Since this holds for an arbitrary $\varepsilon >0$ and 
an arbitrary $\mu\in\mathcal{P}'(T)$ with $\mathrm{bar}(\mu )=O$, 
we have 
$\delta (T, O) \leq\sup_n \delta (T_n,\hspace{0.8mm} O_{n})$. 
\end{proof}

For a $\mathrm{CAT}(1)$ space $X$, 
we prove the following relation between $\delta (\mathrm{Cone}(X) ,O_{\mathrm{Cone}(X)})$ 
and $\tilde{\delta}(X)$ 

\begin{proposition}\label{deltatildedelta}
Let $(X,d_X )$ be a complete $\mathrm{CAT}(1)$ space. 
Then we have 
$$
\delta (\mathrm{Cone}(X),\hspace{0.3mm} O_{\mathrm{Cone}(X)})\leq\tilde{\delta}(X).
$$
\end{proposition}

Before proving Proposition \ref{deltatildedelta}, we establish the following two lemmas. 

\begin{lemma}\label{noconepoint}
Let $(X,d_X )$ be a complete $\mathrm{CAT}(1)$ space. 
Let 
\begin{equation*}
\nu =\sum_{i=1}^m t_i \mathrm{Dirac}_{v_i}
\in\mathcal{P}'(\mathrm{Cone}(X)), 
\end{equation*}
where 
$v_i \in\mathrm{Cone}(X)$ for $i=1,\ldots ,m$ and 
$t_1 ,\cdots ,t_m >0$ with $\sum_{i=1}^m t_i =1$. 
Suppose that $\mathrm{bar}(\nu )=O_{\mathrm{Cone}(X)}$. 
If $v_1 =O_{\mathrm{Cone}(X)}$ and if 
$$
\nu' =\sum_{i=2}^m \frac{t_i}{1-t_1} \mathrm{Dirac}_{v_i} , 
$$
then 
$\mathrm{bar}(\nu' )=O_{\mathrm{Cone}(X)}$ and 
$\delta (\nu )\leq\delta (\nu' )$. 
\end{lemma}

\begin{proof}
The former assertion follows immediately from Lemma \ref{barycenter}. 
Let $\phi' :\mathrm{supp}(\nu' )\to\mathcal{H}$ be a map from 
the support of $\nu'$ to some Hilbert space $\mathcal{H}$ satisfying 
\eqref{umbrellaborn} and \eqref{1-lipschitz} with respect to $\nu'$. 
Define $\phi :\mathrm{supp}(\nu )\to\mathcal{H}$ by 
\begin{align*}
\phi (v_1 )&=0, 
\\
\phi (v_i )&=\phi' (v_i ),\quad i=2,\cdots ,m. 
\end{align*}
Then $\phi$ 
satisfies \eqref{umbrellaborn} and \eqref{1-lipschitz} with respect to $\nu$. 
Moreover, an easy computation shows that 
$$
\frac{\|\int_{\mathrm{Cone}(X)}\phi (v)\nu (dv)\|^2}{
\int_{\mathrm{Cone}(X)}\|\phi (v)\|^2 \nu (dv)}
\leq
\frac{\|\int_{\mathrm{Cone}(X)}\phi' (v)\nu' (dv)\|^2}{
\int_{\mathrm{Cone}(X)}\|\phi' (v)\|^2 \nu' (dv)}. 
$$
Hence, by the definition of $\delta$, the latter assertion follows.
\end{proof}

\begin{lemma}\label{alpha}
Let $(X,d_X )$ be a complete $\mathrm{CAT}(1)$ space 
and let 
$$
\nu =\sum_{i=1}^m t_i \mathrm{Dirac}_{\lbrack x_i ,r_i \rbrack}\in\mathcal{P}'(\mathrm{Cone}(X)),
$$
where $\lbrack x_i ,r_i \rbrack$ is the point on 
$\mathrm{Cone}(X)$ 
represented by $(x_i ,r_i )\in X\times\lbrack 0,\infty )$. 
Suppose that $\alpha >0$, $l\in\{1,2,\cdots ,m-1\}$, and  
$$
\nu'
=
\frac{1}{\sum_{i=1}^l \frac{t_i}{\alpha}+\sum_{i=l+1}^m t_i}
\left(
\sum_{i=1}^l \frac{t_i}{\alpha}\mathrm{Dirac}_{\lbrack x_i ,\alpha r_i \rbrack}
+\sum_{i=l+1}^m t_i \mathrm{Dirac}_{\lbrack x_i ,r_i \rbrack}
\right) .
$$
Then $\mathrm{bar}(\nu' )=O_{\mathrm{Cone}(X)}$ 
if and only if 
$\mathrm{bar}(\nu )=O_{\mathrm{Cone}(X)}$. 
Moreover, if $\mathrm{bar}(\nu )=\mathrm{bar}(\nu' )=O_{\mathrm{Cone}(X)}$ 
and if $\alpha >1$ (resp. $0<\alpha<1$), then 
the inequality 
$\delta (\nu )\leq\delta (\nu' )$ holds 
if and only if 
\begin{equation}
\alpha
\frac{\sum^{l}_{i=1}t_i r_i^2}{\sum_{i=l+1}^m t_i r_i^2}
\leq
\frac{\sum_{i=1}^l t_i}{\sum_{i=l+1}^m t_i}
\quad\left(\textrm{resp.}\hspace{1.6mm}
\alpha
\frac{\sum^{l}_{i=1}t_i r_i^2}{\sum_{i=l+1}^m t_i r_i^2}
\geq
\frac{\sum_{i=1}^l t_i}{\sum_{i=l+1}^m t_i}
\right).
\end{equation}
\end{lemma}

\begin{proof}
The equivalence between 
$\mathrm{bar}(\nu )=O_{\mathrm{Cone}(X)}$ 
and 
$\mathrm{bar}(\nu' )=O_{\mathrm{Cone}(X)}$ 
is an immediate consequence of Lemma \ref{barycenter}. 
Assume that 
$\mathrm{bar}(\nu )=\mathrm{bar}(\nu' )=O_{\mathrm{Cone}(X)}$, 
and 
fix some real Hilbert space $\mathcal{H}$ of dimension $\geq m$.  
Then 
there is a natural bijection $\phi\mapsto\phi'$ between 
the set of all maps from $\mathrm{supp}(\nu )$ to $\mathcal{H}$ 
satisfying \eqref{umbrellaborn} and \eqref{1-lipschitz} with respect to $\nu$, and 
the set of all maps from $\mathrm{supp}(\nu' )$ to $\mathcal{H}$ 
satisfying \eqref{umbrellaborn} and \eqref{1-lipschitz} with respect to $\nu'$: 
it is given by 
\begin{align*}
\phi' \lbrack x_i ,\alpha r_i \rbrack &=\alpha\phi\lbrack x_i ,r_i \rbrack,\quad i=1,\cdots ,l,
\\
\phi' \lbrack x_i ,r_i \rbrack &=\phi \lbrack x_i ,r_i \rbrack ,\quad i=l+1,\cdots ,m .
\end{align*}
Let $\phi :\mathrm{supp}(\nu )\to\mathcal{H}$ and 
$\phi' :\mathrm{supp}(\nu' )\to\mathcal{H}$ be the maps satisfying 
\eqref{umbrellaborn} and \eqref{1-lipschitz} with respect to $\nu$ and $\nu'$ respectively, 
and corresponding to each other under this bijection. 
Let 
$$
T=\frac{1}{\frac{1}{\alpha}\sum_{i=1}^l t_i+\sum_{i=l+1}^m t_i} .
$$
Then we have 
$$
\frac{\|\int_{\mathrm{Cone}(X)}\phi' (p)\nu' (dp)\|^2}{\int_{\mathrm{Cone}(X)}\|\phi' (p)\|^2 \nu' (dp)}
=
T
\frac{\|\sum_{i=1}^m t_i \phi\lbrack x_i ,r_i \rbrack\|^2}{\alpha
\sum_{i=1}^l t_i \|\phi\lbrack x_i ,r_i\rbrack\|^2
+\sum_{i=l+1}^m t_i \|\phi\lbrack x_i ,r_i \rbrack\|^2} .
$$
Hence, 
\begin{multline}\label{one}
\frac{\|\int_{\mathrm{Cone}(X)}\phi' (p)\nu' (dp)\|^2}{\int_{\mathrm{Cone}(X)}\|\phi' (p)\|^2 \nu' (dp)}
-
\frac{\|\int_{\mathrm{Cone}(X)}\phi (p)\nu (dp)\|^2}{\int_{\mathrm{Cone}(X)}\|\phi (p)\|^2 \nu (dp)}
\\
=
\left\|\sum_{i=1}^m t_i \phi\lbrack x_i ,r_i \rbrack\right\|^2 
\frac{T\sum^m_{i=1}t_i r_i^2 -\alpha\sum_{i=1}^l t_i\|\phi\lbrack x_i ,r_i \rbrack\|^2 
-\sum_{i=l+1}^m t_i\|\phi\lbrack x_i ,r_i \rbrack\|^2}{\left( \alpha
\sum_{i=1}^l t_i\|\phi\lbrack x_i ,r_i \rbrack\|^2
+\sum_{i=l+1}^m t_i\|\phi\lbrack x_i ,r_i \rbrack\|^2 \right)
\left(
\sum_{i=1}^m t_i r_i^2
\right)} .
\end{multline}
We also have 
\begin{multline}\label{two}
T\sum^m_{i=1}t_i r_i^2 -\alpha\sum_{i=1}^l t_i\|\phi\lbrack x_i ,r_i \rbrack\|^2 
-\sum_{i=l+1}^m t_i\|\phi\lbrack x_i ,r_i \rbrack\|^2
\\
=
\frac{1-\alpha}{(1-\alpha )\left(\sum_{i=1}^l t_i \right)+\alpha}
\left\{
\alpha (\sum_{i=l+1}^m t_i )
(\sum_{i=1}^l t_i r_i^2 )-
(\sum_{i=1}^l t_i )
(\sum_{i=l+1}^m t_i r_i^2 )
\right\} .
\end{multline}
By \eqref{one} and \eqref{two}, 
the inequality 
$$
\frac{\|\int_{\mathrm{Cone}(X)}\phi' (p)\nu' (dp)\|^2}{\int_{\mathrm{Cone}(X)}\|\phi' (p)\|^2 \nu' (dp)}
\geq
\frac{\|\int_{\mathrm{Cone}(X)}\phi (p)\nu (dp)\|^2}{\int_{\mathrm{Cone}(X)}\|\phi (p)\|^2 \nu (dp)}
$$
holds 
if and only if 
$$
\alpha\geq 1,\quad 
\alpha (\sum_{i=l+1}^m t_i )
(\sum_{i=1}^l t_i r_i^2 )-
(\sum_{i=1}^l t_i )
(\sum_{i=l+1}^m t_i r_i^2 )
\leq
0
$$
or 
$$
0<\alpha\leq 1,\quad 
\alpha (\sum_{i=l+1}^m t_i )
(\sum_{i=1}^l t_i r_i^2 )-
(\sum_{i=1}^l t_i )
(\sum_{i=l+1}^m t_i r_i^2 )
\geq
0 .
$$
The lemma follows easily from this equivalence and 
the bijectivity of the correspondence $\phi\leftrightarrow\phi'$.  
\end{proof}

\begin{proof}[Proof of Proposition \ref{deltatildedelta}]
First suppose that 
$\mu\in\mathcal{P}(\mathrm{Cone}(X))$, 
$\mathrm{bar}(\mu ) =O_{\mathrm{Cone}(X)}$, 
and $\mathrm{supp}(\mu )\subset\iota (X)$. 
Let $\iota :X\to\mathrm{Cone}(X)$ be the canonical inclusion, and 
let $\iota^{-1} :\iota (X)\to X$ be the inverse map. 
Let $\tilde{\phi}:X\to\mathcal{H}$ be a map from $X$ to 
some Hilbert space $\mathcal{H}$ 
satisfying \eqref{realization-tilde}. 
Then the restriction $\phi =\lbrack\tilde{\phi}\circ\iota^{-1}\rbrack|_{\mathrm{supp}(\mu )}$ 
of $\tilde{\phi}\circ\iota^{-1}:\iota (X)\to\mathcal{H}$ 
to $\mathrm{supp}(\mu )$ 
satisfies \eqref{umbrellaborn} and 
\eqref{1-lipschitz}. 
Moreover we have 
$$
\|\mathbb{E}_{\iota^{-1}_* \mu}\lbrack\tilde{\phi}\rbrack\|^2
=
\frac{\|\int_{\mathrm{Cone}(X)}\phi (v)\mu (dv)\|^2}{\int_{\mathrm{Cone}(X)}\|\phi (v)\|^2 \mu (dv)} .
$$
Hence by the definitions of 
$\tilde{\delta}(\iota_*^{-1}\mu )$ and $\delta (\mu )$, we have 
$$
\delta (\mu )\leq\tilde{\delta}(\iota^{-1}_* \mu ).
$$

Thus, if we prove the existence of 
$\nu'\in\mathcal{P}(\mathrm{Cone}(X))$ such that 
\begin{equation}\label{nunu'}
\delta (\nu )\leq\delta (\nu' ),\quad
\mathrm{supp}(\nu' )\subset\iota (X) 
\end{equation}
for any 
$$
\nu
=\sum_{i=1}^m t_i \mathrm{Dirac}_{\lbrack x_i ,r_i \rbrack}\in\mathcal{P}'(\mathrm{Cone}(X))
$$ 
with 
$\mathrm{bar}(\nu )=O_{\mathrm{Cone}(X)}$, 
then the desired assertion follows. 
Here, we can assume $r_i >0$ for all $i\in\{ 1,\cdots ,m\}$ by Lemma \ref{noconepoint}.
And, if 
$r_1 =r_2 =\cdots =r_m$, we can take 
$$
\nu' =\sum_{i=1}^m t_i \mathrm{Dirac}_{\lbrack x_i ,
1 \rbrack} , 
$$
and 
$\nu'$ satisfies \eqref{nunu'} because 
it is straightforward that 
$\delta (\nu)=\delta (\nu' )$. 
So we can assume $r_1 =\cdots =r_l <r_{l+1}\leq \cdots\leq r_m$ without loss of generality. 
Then we have 
$$
\left(\frac{\sum_{i=1}^l t_i}{\sum_{i=l+1}^m t_i}\right)
\slash
\left(\frac{\sum_{i=1}^l t_i r_i^2}{\sum_{i=l+1}^m t_i r_i^2}\right)
\geq
\frac{r_{l+1}^2}{r_1 ^2}
\geq
\frac{r_{l+1}}{r_1} .
$$
Hence, if we set 
$$
\nu_0 =
\frac{1}{\frac{r_1}{r_{l+1}}\sum_{i=1}^{l}t_i +\sum_{i=l+1}^m t_i}
\left(
\sum_{i=1}^{l} \frac{r_1 t_i}{r_{l+1}} \mathrm{Dirac}_{\lbrack x_i ,r_{l+1} \rbrack}
+\sum_{i=l+1}^{m} t_i \mathrm{Dirac}_{\lbrack x_i r_i \rbrack}
\right) ,
$$
then we have 
$$
\delta (\nu_0 )\geq\delta (\nu )
$$
by Lemma \ref{alpha}. 
Repeating this procedure, we finally get 
$$
\nu_1 =\sum_{i=1}^{m} s_i \mathrm{Dirac}_{\lbrack x_i ,r_m \rbrack},
$$
which satisfies 
$
\delta (\nu_1 )\geq\delta (\nu ). 
$
If we set 
$
\nu' =\sum_{i=1}^{m} s_i \mathrm{Dirac}_{\lbrack x_i ,1 \rbrack} ,
$
it is easily seen that 
$
\delta (\nu ')=\delta (\nu_1 )
$, 
and the assertion follows. 
\end{proof}

\section{Proof of the theorem}
Recall that the Gromov-Hausdorff precompactness
is known to be equivalent to the uniformly total boundedness. 
We call the family $\mathcal{X}$ 
of metric spaces 
\textit{uniformly totally bounded} 
if the following two conditions are satisfied: 
\begin{itemize}
\item
There is a constant $D$ such that 
$\mathrm{diam}(X)\leq D$ for all $X\in\mathcal{X}$. 
\item
For any $\varepsilon >0$ there exists $N(\varepsilon )\in\mathbb{N}$ such that 
each $X\in\mathcal{X}$ contains a subset $S_{X,\varepsilon}$ with the following 
property: 
the cardinality of $S_{X,\varepsilon}$ is no greater than $N(\varepsilon )$ and 
$X$ is covered by the union of 
all $\varepsilon$-balls whose centers are in $S_{X,\varepsilon}$. 
\end{itemize}

By Lemma \ref{conedelta}, Lemma \ref{product} and Proposition \ref{deltatildedelta}, 
to prove Theorem \ref{main} it suffices to prove 
the following proposition. 

\begin{proposition}\label{pi/12}
Let $(X,d_X)$ be a complete $\mathrm{CAT}(1)$ space. 
Assume that there exist $N \in\mathbb{N}$ and 
a subset $S=\{x_i \}_{i=1}^{N}\subset X$ such that 
$X$ is covered by the union of 
all $\frac{\pi}{12}$-balls whose centers are in $S$.  
Then there exists a constant $C (N)<1$, 
depending only on $N$, 
such that 
$$
\tilde{\delta}(X)
<
C(N). 
$$
\end{proposition}

\begin{remark}
It follows from the argument in the proof of Proposition \ref{pi/12}, 
we can take 
$$
C(N)
=
\left(
\frac{2}{3}+
\frac{1}{3}
\sqrt{\frac{e^{-\frac{\pi^2}{36N}}+1}{2}}\right)^2 .
$$
as a constant $C(N)$ in the proposition. 
\end{remark}

Before proving Proposition \ref{pi/12}, we will recall a well-known construction 
of a map from a Hilbert space to the unit sphere in 
another Hilbert space, and 
derive some necessary estimates for them. 
We follow Dadarlat and Guentner \cite{DG} to 
explain this construction. 
Let $\mathcal{H}$ be a Hilbert space. 
Let 
$$
\mathrm{Exp}(\mathcal{H})=
\mathbb{R}\oplus\mathcal{H}\oplus (\mathcal{H}\otimes\mathcal{H})
\oplus(\mathcal{H}\otimes\mathcal{H}\otimes\mathcal{H})
\oplus\cdots ,
$$
and define  $\mathrm{Exp}:\mathcal{H}\to\mathrm{Exp}(\mathcal{H})$ 
by 
$$
\mathrm{Exp}(\zeta )
=
1\oplus\zeta\oplus
\left(\frac{1}{\sqrt{2!}}\zeta\otimes\zeta\right)\oplus
\left(\frac{1}{\sqrt{3!}}\zeta\otimes\zeta\otimes\zeta\right)
\oplus\cdots .
$$ 
For $t>0$, 
define a map $G_t$ from $\mathcal{H}$ to 
$\mathrm{Exp}(\mathcal{H})$ to be 
$$
G_t (\zeta )=
e^{-t\|\zeta\|^2}\mathrm{Exp}(\sqrt{2t}\zeta ).
$$ 
Then simple computation shows that 
\begin{equation}\label{Gt}
\cos\angle(G_t (\zeta ),G_t(\zeta' ))
=
\langle G_t (\zeta ), G_t(\zeta' )\rangle
=
e^{-t\|\zeta-\zeta' \|^2}
\end{equation}
for all $\zeta ,\zeta' \in\mathcal{H}$.
In particular, $\|G_t (\zeta)\| =1$ 
for all $\zeta\in\mathcal{H}$. 
Hence we can regard $G_t$ as a map 
from $\mathcal{H}$ to the unit sphere in 
$\mathrm{Exp}(\mathcal{H})$. 

We need the following estimate to prove Proposition \ref{pi/12}. 

\begin{lemma}\label{lipschitz}
Let $(X,d_X )$ be a metric space, and let 
$F:X\to\mathcal{H}$ be an $L$-Lipschitz map ($L>0$) to 
some Hilbert space. 
Suppose that $0<t L^2\leq\frac{1}{2}$.  
Then the map 
$\phi =G_t \circ F :X\to\mathrm{Exp}(\mathcal{H} )$ 
satisfies 
$$
\angle\left( \phi (x),\phi (y)\right)
\leq\min\{\pi ,d_X (x,y)\}
$$
for all $x,y\in X$. 
\end{lemma}

\begin{proof}
By \eqref{Gt} and $L$-Lipschitz continuity of $F$, 
it is sufficient to show that  
\begin{equation}\label{cos}
e^{-tL^2 d_X (x,y)^2}
\geq
\cos\left( \min\{\pi ,d_X (x,y)\}\right)
\end{equation}
for all $x,y\in X$ and all $t\in (0,\frac{1}{2L^2})$. 
When $d_X (x,y)\geq\frac{\pi}{2}$, \eqref{cos} is 
obvious. 
So, if we put $a=tL^2$ and $d=d_X (x,y)$, 
then what we have to show is that 
\begin{equation}\label{ad}
a\leq
\frac{-\log (\cos d)}{d^2}
\end{equation}
holds for any $a\in(0,\frac{1}{2}\rbrack$ and any $d\in\lbrack 0,\frac{\pi}{2} )$. 
But this is obvious 
because the right-hand side of 
\eqref{ad} is non-decreasing with respect to $d$. 
\end{proof}

Now we are ready to prove Proposition \ref{pi/12}. 

\begin{proof}[Proof of Proposition\ref{pi/12}]
First we define a map $F_{S}$ from 
$X$ to $\mathbb{R}^{N}$ by 
$$
F_{S}(x)
=
\left(
d_X (x,x_1),
d_X (x,x_2),
\cdots
d_X (x,x_{N})
\right)
$$
for $x\in X$. 
Then $F_S$ is $\sqrt{N}$-Lipschitz since 
$$
\| F_S (x)-F_S (y)\|
=
\left\{
\sum_{i=1}^N \left( d_X (x,x_i )-d_X (y,x_i)\right)^2
\right\}^{\frac{1}{2}}
\leq
\sqrt{N}\cdot d_X (x,y) .
$$
On the other hand, 
by the definition of the subset $S$, 
for any $x,y\in X$ with $d_X (x,y)\geq\frac{\pi}{3}$, 
there exist  
$i_0 ,i_1\in\{ 1,\cdots N\}$ such that 
\begin{align*}
&d_X (x_{i_0},x)\geq\frac{\pi}{4},\quad d_X (x_{i_0},y)\leq\frac{\pi}{12},\\
&d_X (x_{i_1},y)\geq\frac{\pi}{4},\quad d_X (x_{i_1},x)\leq\frac{\pi}{12} .
\end{align*}
Hence 
\begin{multline}\label{pi/3}
\| F_S (x)-F_S (y)\| \\
\geq
\sqrt{(d_X (x_{i_0},x)-d(x_{i_0},y))^2+(d_X (x_{i_1},x)-d(x_{i_1},y))^2}
\geq \frac{\pi}{3\sqrt{2}} 
\end{multline}
for any $x,y\in X$ with $d_X (x,y)\geq\frac{\pi}{3}$. 

We now set $\phi =G_{\frac{1}{2N}}\circ F_{S}:X\to\mathrm{Exp}(\mathbb{R}^N )$.  
Then the all values of $\phi$ are contained in the unit sphere of $\mathrm{Exp}(\mathbb{R}^N )$, 
and 
$\phi$ satisfies 
$$
\angle\left( \phi (x),\phi (y)\right)
\leq\min\{\pi ,d_X (x,y)\}
$$
for all $x,y\in X$ 
by Lemma \ref{lipschitz}. 
Moreover \eqref{Gt} and \eqref{pi/3} imply that 
\begin{equation}\label{phi}
\angle\left( \phi (x),\phi (y) \right)
\geq
\arccos (e^{-\frac{\pi^2}{36 N}})
\end{equation}
for any $x,y\in X$ with $d_X (x,y)\geq\frac{\pi}{3}$. 

Set $\eta =\arccos (e^{-\frac{\pi^2}{36 N}})$, and  
let $\mu$ 
be an arbitrary measure in $\mathcal{P}(X)$ 
with 
$\mathrm{bar}(\iota_{*}\mu )=O_{\mathrm{Cone}(X)}$, where 
$\iota :X\to\mathrm{Cone}(X)$ is the canonical inclusion and 
$O_{\mathrm{Cone}(X)}$ is the cone point of $\mathrm{Cone}(X)$. 
Then we have 
\begin{equation}\label{eta}
\phi_{*}\mu \left( B\left( v,\frac{\eta}{2}\right)\right)
\leq\frac{2}{3}
\end{equation}
for 
any point $v$ on the unit sphere in 
$\mathrm{Exp}(\mathbb{R}^N )$, 
where 
$$
B\left(v,\frac{\eta}{2}\right)
=
\left\{ u\in\mathrm{Exp}(\mathbb{R}^N )\hspace{1mm}\Big|\hspace{1mm}
\| u\|=1,\hspace{1mm}\angle (v,u )<\frac{\eta}{2}
\right\} .
$$
This is because if there exists some vector $\phi (x_0 )$ contained in 
$B\left(v,\frac{\eta}{2}\right)\cap\phi (X)$, then 
by \eqref{phi} and Corollary \ref{pi/3-2/3} we have 
\begin{align*}
\phi_{*}\mu \left( B\left( v,\frac{\eta}{2}\right)\right)
&\leq
\phi_{*}\mu \left( B\left(\phi(x_0 ),\eta\right)\right)
\\
&=
\mu\left(\phi^{-1}\left(
B\left( \phi (x_0 ),\eta\right)\right)
\right)
\\
&\leq
\mu\left(
B\left( x_0 ,\frac{\pi}{3}\right)
\right)
\leq
\frac{2}{3},
\end{align*}
where $B\left( x_0 ,\frac{\pi}{3}\right)$ is the 
open ball in $X$ centered at 
$x_0$ with radius $\frac{\pi}{3}$. 
In the case $B\left(v,\frac{\eta}{2}\right)\cap\phi (X)=\phi$, 
\eqref{eta} obviously holds. 

By \eqref{eta}, we have
\begin{align*} 
\int_X
\langle v, \phi (x)\rangle\mu (dx)
&=
\int_{\mathcal{S}}\langle v,u\rangle \phi_* \mu (du)
\\
&=
\int_{B(v,\frac{\eta}{2})}
\langle v,u\rangle \phi_* \mu (du)
+
\int_{\mathcal{S}\backslash B(v,\frac{\eta}{2})}
\langle v,u\rangle \phi_* \mu (du)
\\
&\leq
1\times
\phi_* \mu \left( B\left( v,\frac{\eta}{2}\right)\right)
+
\cos\frac{\eta}{2}\times
\left\{
1-\phi_* \mu \left( B\left( v,\frac{\eta}{2}\right)\right)
\right\}
\\
&\leq
1\times\frac{2}{3}+
\left(\cos\frac{\eta}{2}\right)\times\frac{1}{3}
,
\end{align*}
where $\mathcal{S}$ is the unit sphere in $\mathrm{Exp}(\mathbb{R}^N)$.
Setting $v=\tilde{\mathbb{E}}_{\mu}\lbrack\phi\rbrack$ in the 
above inequality and 
using \eqref{variance}, we have 
$$
\|\mathbb{E}_{\mu}\lbrack\phi\rbrack\|
=
\left\|
\int_{X}\langle\tilde{\mathbb{E}}_{\mu}\lbrack\phi\rbrack ,
\phi (x) \rangle \mu (dx)
\right\|
\leq 
c_{N}, 
$$
where
$$
c_{N}
=
1\times\frac{2}{3}+
\left(\cos\frac{\eta}{2}\right)\times\frac{1}{3}
=
\frac{2}{3}+
\frac{1}{3}
\sqrt{\frac{e^{-\frac{\pi^2}{36N}}+1}{2}}
$$
Thus, by the definition of $\tilde{\delta}(X)$, 
$$
\tilde{\delta}(X)
\leq
c_N^2 <1
$$
which proves the proposition. 
\end{proof}

Finally, we remark that the proof of 
Proposition \ref{pi/12} works for 
the following more general statement.  

\begin{proposition}\label{sufficient}
Let $0<\theta<\frac{\pi}{2}$, $0<\alpha <1$ and $\varepsilon >0$. 
Let $(X,d_X)$ be a complete $\mathrm{CAT}(1)$ space. 
Assume that there exists a finite subset 
$S\subset X$ such that  
$$
\#\left\{ s\in S\hspace{1mm}\big|\hspace{1mm} \|d_{X}(x,s)-d_{X}(y,s)\|\geq\varepsilon\right\} \geq\alpha \# S 
$$
whenever $x,y\in X$ and $d(x,y)\geq\theta$. 
Here, $\#S$ stands for the cardinality of $S$. 
Then there exists a constant $C=C (\theta, \alpha ,\varepsilon)<1$ 
such that 
$$
\tilde{\delta}(X)
\leq
C. 
$$
\end{proposition}

\begin{proof}
We denote the cardinality of $S$ by $N$. 
Let $F_S$ be the map from $X$ to $\mathbb{R}^N$ as in the proof of Proposition \ref{pi/12}
with respect to our set $S$. 
Then $F_S$ is $\sqrt{N}$-Lipschitz and we have 
\begin{equation}\label{theta}
\| F_S (x)-F_S (y)\| \\
\geq
\sqrt{\alpha N}\varepsilon
\end{equation}
for any $x,y \in X$ with $d_X (x,y)\geq\theta$. 
If we set $\phi =G_{\frac{1}{2N}}\circ F_{S}:X\to\mathrm{Exp}(\mathbb{R}^N )$, 
then all the values of $\phi$ are contained in the unit sphere of $\mathrm{Exp}(\mathbb{R}^N )$, 
and 
$\phi$ satisfies 
$$
\angle\left( \phi (x),\phi (y)\right)
\leq\min\{\pi ,d_X (x,y)\}
$$
for all $x,y\in X$ 
by Lemma \ref{lipschitz}. 
Moreover \eqref{Gt} and \eqref{theta} imply that 
\begin{equation}\label{phi2}
\angle\left( \phi (x),\phi (y) \right)
\geq
\arccos (e^{-\frac{\alpha\varepsilon^2}{2}})
\end{equation}
for any $x,y\in X$ with $d_X (x,y)\geq\theta$. 

Now the rest of the proof is done exactly in the same manner as in the proof of Proposition \ref{pi/12}, 
and we have 
$$
\tilde{\delta}(X)
\leq 
(c_{\theta ,\alpha ,\varepsilon})^2, 
$$
where
\begin{align*}
c_{\theta ,\alpha ,\varepsilon}
&=
1\times\frac{1}{1+\cos\theta}+
\left(\cos\frac{\arccos (e^{-\frac{\alpha\varepsilon^2}{2}})}{2}\right)\times
\left( 1-\frac{1}{1+\cos\theta}\right) 
\\
&=
\frac{1}{1+\cos\theta}+
\sqrt{\frac{e^{-\frac{\alpha\varepsilon^2}{2}}+1}{2}}\times
\frac{\cos\theta}{1+\cos\theta}  <1.
\end{align*}
\end{proof}

\begin{acknowledgement}
I would like to thank Professor S. Nayatani, Professor K. Fujiwara and Dr. T. Kondo 
for many helpful discussions. 
\end{acknowledgement}


\begin{thebibliography}{}
%
%
\bibitem{BH}
\textsc{M. R. Bridson and A Haefliger}, 
\textrm{Metric spaces of non-positive curvature}, 
Springer-Verlag, Berlin, Heidelberg, 1999

\bibitem{BBI}
\textsc{D. Burago, Y. Burago and S. Ivanov}, 
\textrm{A course in metric geometry}, 
Graduate studies in Math. 33 AMS, Providence, RI, 2001. 

\bibitem{DG}
\textsc{M. Dadarlat and E. Guentner}, 
\textrm{Constructions preserving Hilbert space uniform 
embeddability of discrete groups}, 
Vol. 355, Proc. Amer. Math. Soc. No. 8 (2003) 3253--3275.

\bibitem{G}
\textsc{M. Gromov},  
\textrm{Random walks in random groups},
Geom. Funct. Anal. \textbf{13} (2003) 73--146.

\bibitem{IN}
\textsc{H. Izeki and S. Nayatani}, 
\textrm{Combinatorial harmonic maps and discrete-group actions on Hadamard 
spaces},
Geom. Dedicata \textbf{114} (2005) 147--188.


\bibitem{IKN}
\textsc{H. Izeki, T. Kondo, S. Nayatani}, 
\textrm{Fixed-point property of random groups},
Ann. Global Anal. Geom. \textbf{35} (2009), 363--379.




\bibitem{IKN2}
\textsc{H. Izeki, T. Kondo, S. Nayatani}, 
\textrm{$N$-step energy of maps and fixed-point property of random groups}, 
preprint.

\bibitem{K}
\textsc{T. Kondo}, 
\textrm{Fixed-point property for CAT(0) spaces},
preprint.

\bibitem{K2}
\textsc{T. Kondo}, 
\textrm{CAT(0) spaces and expanders}, 
preprint. 

\bibitem{P}
\textsc{P. Pansu}, 
\textrm{Superrigidit\'e geometrique et applications harmoniques},
S\'eminaires et congr\`es \textbf{18}, 
375--422, Soc. Math. France, Paris (2008)

\bibitem{S}
\textsc{K. T. Sturm}, 
\textrm{Probability measures on metric spaces of nonpositive curvature},
In hHeat Kernels and Analysis on Manifolds, Graphs, and Metric Spacesh (edited by
P. Auscher, T. Coulhon, A. Grigorfyan) 357-390. Contemporary Mathematics \textbf{338}.
AMS 2003.

\bibitem{To}
\textsc{T. Toyoda}, 
\textrm{Continuity of a certain invariant of a measure on a CAT(0) space},
Nihonkai Math. J. 20(2009), 89--97.
\end{thebibliography}
%

\end{document}